\documentclass[11pt]{amsart}
\usepackage[linkcolor=blue, citecolor = blue]{hyperref}
\hypersetup{colorlinks=true}

\usepackage{amsthm,amsfonts,amsmath,amssymb,amscd} 

\usepackage{verbatim}
\usepackage{float}
\usepackage{wrapfig}
\usepackage{mathtools}
\usepackage[color=orange!20, linecolor=orange]{todonotes}

\usepackage{enumitem}

\usepackage{ytableau}

\usepackage{tikz}
\usetikzlibrary{decorations.markings}
\tikzstyle{vertex}=[circle, draw, inner sep=0pt, minimum size=4pt]

\usetikzlibrary{arrows,automata}
\usepackage{tikz, tikz-3dplot, pgfplots}
\usetikzlibrary[positioning,patterns]

\newtheorem{theorem}{Theorem}%[section]

\newtheorem{lemma}[theorem]{Lemma}
\newtheorem{corollary}[theorem]{Corollary}
\newtheorem{conjecture}[theorem]{Conjecture}

\theoremstyle{definition}

\theoremstyle{remark}
\newtheorem{remark}[theorem]{Remark}

\newcommand{\sgn}{\mathrm{sgn}}

\newcommand{\srank}{\mathrm{slice}\text{-}\mathrm{rank}}

\title[Saturation of Rota's basis conjecture]
{Saturation of Rota's basis conjecture} 

\author[Damir Yeliussizov]{Damir Yeliussizov}
\address{KBTU, Almaty, Kazakhstan}
\email{\href{mailto:yeldamir@gmail.com}{yeldamir@gmail.com}}

\begin{document}
\begin{abstract}
We prove an asymptotic saturation-type version of Rota's basis conjecture. It relies on the connection of Tao's slice rank with unstable tensors from geometric invariant theory. 
\end{abstract}

\maketitle

%\newpage
%\tableofcontents

%\newpage
\section{Introduction}
The basis conjecture, stated as Problem 1 in Rota's article 
``Ten mathematics problems I will never solve"  %by G.~C. Rota 
\cite{rot10}, is the following.
%We assume $V = \mathbb{C}^n$.
Let $V$ be an $n$-dimensional vector space over a field. 
\begin{conjecture}
\label{rota}
Let $B_1,\ldots, B_n$ be $n$ bases of $V$. There is $n \times n$ matrix $A$ such that: 
\begin{itemize}
\item  in the $i$-th row of $A$ each element of $B_i$ appears one time, for $i =1, \ldots, n$ 
\item every column of $A$ forms a basis of $V$.\footnote{To be precise, each entry of $A$ is a vector in $V$.}
\end{itemize}
\end{conjecture} 

%The main result of this paper is the following 
In this paper we prove the following %`saturation' 
asymptotic %`saturation'
version of the basis conjecture %of the basis conjecture 
%of Conjecture~\ref{rota} 
for $V = \mathbb{C}^n$. %, which presents the problem %basis conjecture %Rota's problem 
%as a saturation-type\footnote{By analogy with algebraic notions of saturation for monoids or ideals.} question.

\begin{theorem}
\label{one}
Let $B_1,\ldots, B_n$ be $n$ bases of $V$. There is $\ell \ge 1$ and $n \times \ell n$ matrix $A$ such that: 
\begin{itemize}
\item in the $i$-th row of $A$ each element of $B_i$ appears $\ell$ times, for $i =1, \ldots, n$ 
\item every column of $A$ forms a basis of $V$. 
\end{itemize}
\end{theorem} 

This result presents the problem (or can be viewed) 
as a saturation-type\footnote{By analogy with algebraic notions of saturation for monoids or ideals.} question. The proof of Theorem~\ref{one} 
uses recent ideas from tensors and invariant theory. 
It relies on the connection of {\it Tao's slice rank} \cite{tao} with {\it unstable tensors} from geometric invariant theory as developed in \cite{blasiak, widg}. 
We combine this theory with a method similar to \cite{onn} 
relating
%showing how (for even $n$) 
the Alon--Tarsi conjecture \cite{at} with Rota's basis conjecture. 

\vspace{0.7em}

Let us briefly summarize some known results on Rota's basis conjecture. 
It is often stated more generally for matroids and is related to several other conjectures \cite{hr}, see also \cite{chow1}.   
The Alon--Tarsi conjecture on latin squares \cite{at}, which is proved for specific %infinitely many specific 
$n = p \pm 1$ where $p > 2$ is any prime \cite{dri1,  glynn}, implies Conjecture~\ref{rota} \cite{hr, onn}. 
The conjecture is also known to hold for some special classes of matroids, such as 
strongly base-orderable \cite{wild}, paving~\cite{gh}, and matroids of rank at most~$4$ \cite{che}. 
Lower bounds on the number of disjoint transversal\footnote{A  transversal is referred here to a set with at most one element from each $B_i$.} bases %(for matroids) 
were obtained 
in \cite{gw, dg}, in \cite{bkps} the bound of $(1/2-o(1))n$ such bases is proved, and in \cite{fm} the bound of $n - o(n)$ bases is given for matroids of large girth. In
\cite{ab} it is shown that $\cup B_i$ can be decomposed into $2n$ transversal independent sets, which was  improved to $2n-2$ in \cite{pm}. In \cite{pok} the result of $n - o(n)$ disjoint transversal independent sets of sizes $n - o(n)$ is obtained.
For matroids, the known methods on the problem are mostly combinatorial. 
For the Alon--Tarsi conjecture, there is an algebraic approach using hyperdeterminants \cite{al}. 
Rota's basis conjecture was also the topic of the Polymath~12 project~\cite{polymath, pm}. 

%\newpage

\section{Tensors}\label{prelim}
We use the notation $[n] := \{ 1,\ldots, n\}$ and $G^d := G \times \cdots \times G$ ($d$ times) for a group $G$.
Let $V = \mathbb{C}^n$. %and we use the notation $[n] := \{ 1,\ldots, n\}$.
{\it Tensors} are elements of the space $V^{\otimes d} = V \otimes \cdots \otimes V$ ($d$ times).  Each tensor of $V^{\otimes d}$ can be represented in coordinates as 
$$
\sum_{1 \le i_1,\ldots, i_d \le n} T(i_1,\ldots, i_d)\, \mathbf{e}_{i_1} \otimes \cdots \otimes \mathbf{e}_{i_d},
$$
where $T : [n]^d \to \mathbb{C}$ which we call a {\it $d$-tensor}, and $(\mathbf{e}_i)$ is the standard basis of $V$. We denote by $\mathsf{T}^d(n) := \{T : [n]^d \to \mathbb{C} \}$ the set of $d$-tensors.

\vspace{0.5em}

%\begin{definition}[Multilinear product] 
%\subsection{} 
Let $A_1,\ldots A_d \in \mathsf{T}^2(n)$ be $n \times n$ matrices and $X \in \mathsf{T}^d(n)$ be a $d$-tensor. The {\it multilinear product} is defined as follows
$$
(A_1,\ldots, A_d) \cdot X = Y \in \mathsf{T}^d(n),
$$
where 
$$
Y(i_1,\ldots, i_d) = \sum_{j_1,\ldots, j_d \in [n]} A(i_1,j_1) \cdots A(i_d, j_d)\, X(j_1,\ldots, j_d).
$$
The multilinear product defines the natural $\mathrm{GL}(V)^d$ action on $\mathsf{T}^d(n)$, 
%For $A_1,\ldots, A_d \in \mathrm{GL}(V)$, multilinear product 
and simply expresses change of bases of $V$ for a tensor. Note that for matrices $B_1,\ldots, B_d \in \mathsf{T}^2(n)$ we have 
$$
(A_1 B_1, \ldots, A_d B_d) \cdot X = (A_1,\ldots, A_d) \cdot ((B_1,\ldots, B_d) \cdot X).
$$
%\end{definition}

\vspace{0.5em}

%\begin{definition}[Tensor product]
%\subsection{} 
The {\it tensor product} %$X \otimes Y$ 
of $X \in \mathsf{T}^{d}(n), Y \in \mathsf{T}^d(m)$ is defined as $T = X \otimes Y \in \mathsf{T}^d(nm)$ given by 
$$
T(k_1,\ldots, k_d) = X(i_1,\ldots, i_d) \cdot Y(j_1,\ldots, j_d), \quad k_{\ell} = i_{\ell} (m-1) + j_{\ell}.
$$
Alternatively, we can view the $\ell$-th coordinate of $T$ as a pair $(i_{\ell}, j_{\ell}) \mapsto k_{\ell}$ ordered lexicographically, for $\ell \in [d]$. For $X \in \mathsf{T}^d(n)$, the tensor  $X^{\otimes k} = X \otimes \cdots \otimes X \in \mathsf{T}^d(n^k)$ denotes the $k$-th tensor power of $k$ copies of $X$.
%\end{definition}

\section{The slice rank}
%\begin{definition}[\cite{tao}]
A nonzero $d$-tensor $T \in \mathsf{T}^d(n)$ has {\it slice rank} $1$ if it can be decomposed in a form
$$
T(i_1,\ldots, i_d) = \mathbf{v}(i_k) \cdot T_1(i_1,\ldots, i_{k-1}, i_{k+1}, \ldots, i_d),
$$
for some $k \in [d]$, a vector $\mathbf{v} \in V$ and a $(d-1)$-tensor $T_1 \in \mathsf{T}^{d-1}(n)$. The {\it slice rank} of $T \in \mathsf{T}^d(n)$, denoted by $\srank(T)$, is then the minimal $r$ such that 
$$
T = T_1 + \ldots + T_r,
$$
where each summand $T_{i}$ has slice rank $1$. (Note that each $T_i$ can be decomposed differently and along different coordinates $k$.) 
%\end{definition}

For $T \in \mathsf{T}^d(n)$ we have the inequality
$$
\srank(T) \le n,
$$
since $T$ can always be expressed as the sum of slice rank 1 tensors as follows
$$
T(i_1,\ldots, i_d) = \sum_{\ell = 1}^n \delta(i_1,\ell) \cdot T(\ell, i_2,\ldots, i_d),
$$
where $\delta$ is the Kronecker delta function.

The following lemma is useful for finding the slice rank of certain sparse tensors. 
\begin{lemma}[\cite{st}]\label{saw}
Equip the set $[n]$ with $d$ total orderings $\le_i$ for $i \in [d]$, which define the partial order $\le$ on $[n]^d$. Let $T \in \mathsf{T}^d(n)$ whose support $\Gamma = \{(i_1,\ldots, i_d) : T(i_1,\ldots, i_d) \ne 0\}$ is an antichain w.r.t. $\le$. Then 
$$
\srank(T) = \min_{\Gamma = \Gamma_1 \cup \cdots \cup \Gamma_d } |\pi_1(\Gamma_1)| + \ldots + |\pi_d(\Gamma_d)|,
$$
where the minimum is over set partitions $\Gamma = \Gamma_1 \cup \cdots \cup \Gamma_d $ and $\pi_i : [n]^d \to [n]$ is the projection map on the $i$-th coordinate. 
\end{lemma}

\begin{remark}
The slice rank was introduced by Tao in \cite{tao} and studied in \cite{st}. 
This notion found many applications especially in additive combinatorics, see \cite{gro} for a related survey. 
\end{remark}

\begin{remark}
For $d = 2$, the slice rank coincides with the usual matrix rank. For $d \ge 3$, it significantly differs from the more common {\it tensor rank} (e.g. \cite{lands}) which can be way larger. 
\end{remark}

\section{The Levi--Civita tensor}
%\begin{definition}[Levi--Civita tensor]
For $i_1,\ldots, i_n \in [n]$ 
the {\it Levi--Civita symbol} is defined as follows
$$
\varepsilon(i_1,\ldots, i_n) := 
\begin{cases}
	\sgn(i_1,\ldots, i_n), & \text{ if $(i_1,\ldots, i_n) \in S_n$ is a permutation},\\
	0, & \text{ otherwise}.
\end{cases}
$$
The {\it Levi--Civita tensor} $E_n \in \mathsf{T}^n(n)$ is the $n$-tensor given by 
$
E_n(i_1,\ldots, i_n) = \varepsilon(i_1,\ldots, i_n)
$.
%\end{definition}

\begin{lemma}\label{five}
We have: $\srank(E_n^{\otimes k}) = n^k$ is full for all $k$. 
\end{lemma}
\begin{proof}
%Let us show that the support of 
%Let $X = E_n^{\otimes k} \in \mathsf{T}^{n}(n^k)$. %is an antichain w.r.t. a lexicographic ordering of 
The support of $E_n^{\otimes k} \in \mathsf{T}^{n}(n^k)$ can be identified with the following set
$$
\Gamma = \left\{(\mathbf{i}_1,\ldots, \mathbf{i}_n)\, :\,  \mathbf{i}_{\ell} = (i_{\ell, 1}, \ldots, i_{\ell, k}) \in [n]^k \text{ for } \ell \in [n], \text{ and } (i_{1,j}, \ldots, i_{n,j}) \in S_n \text{ for } j \in [n]\right\}.
$$
Take the lexicographic ordering $\le_{\ell}$ on $\mathbf{i}_{\ell} \in [n]^k$ for each $\ell \in [n]$, which define the partial order $\le$ on $\Gamma$. Let us show that $\Gamma$ is an antichain w.r.t. this partial order. Assume we have $(\mathbf{i}_1,\ldots, \mathbf{i}_n) \le (\mathbf{i}'_1,\ldots, \mathbf{i}'_n)$ for elements of $\Gamma$, which means $\mathbf{i}_{\ell} = (i_{\ell, 1}, \ldots, i_{\ell, k})\, \le_{\ell}\, \mathbf{i}'_{\ell} = (i'_{\ell, 1}, \ldots, i'_{\ell, k})$ for all $\ell \in [n]$. In particular, $i_{\ell, 1} \le i'_{\ell, 1}$ for all $\ell \in [n]$ but both $(i_{1,1}, \ldots, i_{n,1}), (i'_{1,1}, \ldots, i'_{n,1}) \in S_n$ are permutations which is only possible when $(i_{1,1}, \ldots, i_{n,1}) = (i'_{1,1}, \ldots, i'_{n,1})$. Since $\le_{\ell}$ are lexicographic, we then have $i_{\ell, 2} \le i'_{\ell, 2}$ for all $\ell \in [n]$ and by the same argument we get $(i_{1,2}, \ldots, i_{n,2}) = (i'_{1,2}, \ldots, i'_{n,2})$. Proceeding the same way we obtain that $(i_{1,j}, \ldots, i_{n,j}) = (i'_{1,j}, \ldots, i'_{n,j})$ for all $j \in [n]$ and hence $(\mathbf{i}_1,\ldots, \mathbf{i}_n) = (\mathbf{i}'_1,\ldots, \mathbf{i}'_n)$ which shows that $\Gamma$ is indeed an antichain. 

%\newpage
Let $\rho : [n]^k \to [n]^k$ be the (bijective) cyclic shift map given by 
$$
\rho : (i_1,\ldots, i_k) \longmapsto (i'_1,\ldots, i'_k) = (i_1 + 1,\ldots, i_k + 1) \text{ mod $n$}. 
$$
Consider the following subset of $\Gamma$
$$
S = \left\{(\mathbf{i}, \rho\, \mathbf{i}, \ldots, \rho^{n-1} \mathbf{i}) : \mathbf{i} \in [n]^k \right\} \subset \Gamma.
$$
Take any partition $\Gamma = \Gamma_1 \cup \cdots \cup \Gamma_n$. 
%Let $a_j = |\Gamma_j \cap S|$ for $j \in [n]$. 
Note that for each $j \in [n]$ we have $|\pi_j(\Gamma_j)| \ge |\Gamma_j \cap S|$ since the elements of $S$ differ in the $j$-th coordinate. %for each $j \in [n]$. 
Hence we have 
$$
|\pi_1(\Gamma_1)| + \ldots + |\pi_n(\Gamma_n)| \ge |\Gamma_1 \cap S| + \ldots + |\Gamma_n \cap S| = |S| = n^k,
$$
which by Lemma~\ref{saw} implies that $\srank(E_n^{\otimes k}) \ge n^k$. 
On the other hand, we know that $\srank(E_n^{\otimes k}) \le n^k$ and hence the equality follows.
%Since $\Gamma$ contains $n^k$ elements the first coordinate 
\end{proof}

\begin{remark}
It was noticed in \cite{gow} that $\srank(E_3) = 3$.
\end{remark}

\section{Unstable tensors}
%\begin{definition}
The notion of unstable tensors comes from geometric invariant theory \cite{mumf}. A tensor $X \in \mathsf{T}^d(n)$ is called {\it unstable} if $P(X) = 0$ for every  $\mathrm{SL}(n)^d$-invariant homogeneous polynomial~$P$.
%\end{definition}
A tensor which is {\it not} unstable is called {\it semistable}.
The following characterization of unstable tensors shows their connection with the slice rank.

\begin{theorem}[{\cite[Cor.~6.5]{widg}}]
%We have: 
A tensor $X \in \mathsf{T}^d(n)$ is unstable iff $\srank(X^{\otimes k}) < n^k$ is not full for some $k$. %, where $T^{\otimes k} = T \otimes \cdots \otimes T$ ($k$ times) viewed as a tensor in $\mathsf{T}^d(n^k)$.
\end{theorem}

Lemma~\ref{five} with this Theorem now give the following result. %which is important for the main proof.
\begin{corollary}
The Levi--Civita tensor $E_n$ is semistable. %\footnote{A tensor which is not unstable is also called {\it semistable}.}
\end{corollary}

%Let $G = \mathrm{SL}(n, \mathbb{C})^d := \mathrm{SL}(n, \mathbb{C}) \times \cdots \times \mathrm{SL}(n, \mathbb{C})$ ($d$ times).

We use the following concrete description of SL-invariant generating polynomials.

\begin{lemma}[{\cite[Prop.~3.10]{widg}, cf. \cite[Ex.~7.18]{widg2}}]
The space of $\mathrm{SL}(n)^d$-invariant homogeneous polynomials of degree $M$ is nonzero only if $M$ is divisible by $n$, in which case it is spanned by the polynomials $\{ P_{M, \vec{\pi}} \} $ indexed by $d$-tuples of permutations $\vec \pi = (\pi_1,\ldots, \pi_d) \in (S_M)^d$
and given by 
\begin{align}\label{pmx}
P_{M, \vec{\pi}}(X) = \sum_{J_1,\ldots, J_d\, :\, [M] \to [n]}\, \prod_{k = 1}^d \varepsilon(J_k \circ \pi_k) \prod_{i = 1}^M X(J_1(i), \ldots, J_d(i)), \quad X \in \mathsf{T}^d(n), 
\end{align}
where for a map $J : [M] \to [n]$ we define the sign
$$
\varepsilon(J) := %\prod_{k = 1}^{M/n} \varepsilon(J(k n - n + 1), \ldots, J(k n)) 
\varepsilon(J(1), \ldots, J(n)) \cdot \varepsilon(J(n+1), \ldots, J(2n)) \cdot \ldots \cdot \varepsilon(J(M - n + 1), \ldots, J(M)) 
\in \{0, \pm 1 \}.
$$
\end{lemma}

\begin{corollary}\label{cor1}
Let $X \in \mathsf{T}^d(n)$ be semistable. Then $P_{M, \vec{\pi}}(X) \ne 0$ for some $M$ divisible by $n$ and permutations $\vec \pi \in (S_M)^d$.
\end{corollary}

It is also helpful to use the polynomials $P$ as relative $\mathrm{GL}$-invariants. 
\begin{lemma}\label{pgl}
Let $X \in \mathsf{T}^d(n)$ and $A_1,\ldots, A_d \in \mathrm{GL}(n)$. We have 
\begin{align*}
P_{M, \vec{\pi}}((A_1,\ldots, A_d) \cdot X) = P_{M, \vec{\pi}}(X) \cdot \det(A_1)^{M/n} \cdots \det(A_d)^{M/n}.
\end{align*}
\end{lemma}
\begin{proof}
It is enough to check the identity for one matrix $A = A_1$. Write $A = BD$ for $B \in \mathrm{SL}(n)$ and $D = \mathrm{diag}(\det(A), 1, \ldots, 1)$. Then as $P_{M, \vec{\pi}}$ is $\mathrm{SL}(n)^d$-invariant, we get
$$
P_{M, \vec{\pi}}((BD,I, \ldots, I) \cdot X) = P_{M, \vec{\pi}}((B, I, \ldots, I) \cdot ((D,I, \ldots, I) \cdot X) )= P_{M, \vec{\pi}}((D,I, \ldots, I) \cdot X).
$$
Let 
$Y = (D,I, \ldots, I) \cdot X$. We have 
$$Y(i_1,\ldots, i_d) = \sum_{j} D(i_1, j) X(j, i_2,\ldots, i_d) = 
\begin{cases}
	\det(A) \cdot X(i_1, \ldots, i_d), & \text{ if $i_1 = 1$},\\
	X(i_1, \ldots, i_d), & \text{ otherwise}.
\end{cases}
$$
From the formula \eqref{pmx} we can see that each nonzero term 
$\prod_{k = 1}^d \varepsilon(J_k \circ \pi_k) \prod_{i = 1}^M X(J_1(i), \ldots, J_d(i))$
of $P_{M, \vec{\pi}}(X)$ has exactly $M/n$ variables $X(1, * \ldots, *)$. Hence,
$P_{M, \vec{\pi}}(Y) = P_{M, \vec{\pi}}(X) \cdot \det(A)^{M/n}$ as needed.
\end{proof}

\begin{remark}
Connection of slice rank with unstable tensors was first established in \cite{blasiak}, where it was shown that $\srank(X) < n$ implies $X$ is unstable, and if $X$ is unstable then $\srank(X^{\otimes k}) < n^k$ for some $k$. In \cite{blasiak} these results are given for $d = 3$ and for any $d$ the statements are in \cite{widg}; the proofs use the Hilbert--Mumford criterion.
\end{remark}

\begin{remark}
The formula \eqref{pmx} is given in exactly this form in \cite[Ex.~7.18]{widg2}, and in \cite[Prop.~3.10]{widg} it is stated in a slightly  different form. 
\end{remark}

\begin{remark}\label{lbd}
%Let us also say that 
The degree $M$ can be bounded above using a result from \cite{derk}, see \cite[Lemma~7.11]{widg} for a precise statement, which gives $M \le d^{d n^2 - d} n^d$.
\end{remark}

\begin{remark}
For even $d$, the minimal degree $n$ SL-invariant polynomial $P_{n, \vec \pi}$
is in fact Cayley's first hyperdeterminant \cite{cay}. In \cite{ay} it is shown that hyperdeterminants also vanish on tensors whose certain refinements of the slice rank are non-full. 
\end{remark}

\section{Determinantal tensors} For a matrix $A$ denote by $A[i]$ the $i$-th column vector of $A$.
For matrices $A_1,\ldots, A_n \in \mathrm{GL}(n)$ define the {\it determinantal tensor} $D = {D}(A_1,\ldots, A_n) \in \mathsf{T}^n(n)$ given by 
$$
D(i_1,\ldots, i_n) := \det(A_1[i_1], \ldots, A_n[i_n]), \quad \forall i_1,\ldots, i_n \in [n].
$$

\begin{lemma}\label{16}
We have:

(i)
Let $A_1,\ldots, A_n, B_1,\ldots, B_n \in \mathrm{GL}(n)$.  Then
$$
D(A_1 B_1,\ldots, A_n B_n) = (B^T_1,\ldots, B^T_n) \cdot D(A_1,\ldots, A_n).
$$

(ii) $
D(I_n, \ldots, I_n) = E_n,
$
where $I_n$ is the identity $n \times n$ matrix.

\end{lemma}
\begin{proof}
(i) It is enough to check the identity for one matrix $B_1 = B$. %Let $T = D(A_1 B, A_2,\ldots, A_n)$. 
By definition and multilinearity of determinants we have 
\begin{align*}
D(A_1 B, A_2,\ldots, A_n) (i_1,\ldots, i_n) &= \det(A_1 B [i_1], A_2[i_2], \ldots, A_n[i_n]) \\
&= \det\left( \sum_{j = 1}^n A_1^{}[j] \cdot  B(j, i_1), A_2[i_2], \ldots, A_n[i_n] \right)\\
&= \sum_{j = 1}^n B(j, i_1) \cdot \det(A_1[j], A_2[i_2], \ldots, A_n[i_n])\\
&= (B^T, I_n, \ldots, I_n) \cdot D(A_1,\ldots, A_n) (i_1,\ldots, i_n).
\end{align*}
(ii) We have
$$
D(I_n, \ldots, I_n) (i_1,\ldots, i_n) = \det(\mathbf{e}_{i_1}, \ldots, \mathbf{e}_{i_n}) = \varepsilon(i_1,\ldots, i_n)
$$
and the equality follows.
\end{proof}

\begin{corollary}
Let $B_1,\ldots, B_n \in \mathrm{GL}(n)$. We have 
$$
D(B_1,\ldots, B_n) = (B^T_1,\ldots, B^T_n) \cdot D(I_n,\ldots, I_n) = (B^T_1,\ldots, B^T_n) \cdot E_n.
$$
\end{corollary}

\begin{remark}
Determinantal tensors are implicitly used in \cite{onn}; an explicit formulation appears in \cite{al}. 
\end{remark}

\section{Proof of Theorem~\ref{one}}
%\subsection{Levi--Civita tensor}
%Let us first show that the Levi--Civita tensor is not unstable.
%\subsection{Proof of Theorem~\ref{one}} 
We are now ready to prove the result. We have $B_1, \ldots, B_n \in \mathrm{GL}(n)$ whose elements are given by column vectors. %viewed as $n \times n$ matrices. 
Consider the determinantal tensor 
$$
D = D(B_1,\ldots, B_n) = (B^T_1,\ldots, B^T_n) \cdot E_n.
$$
Since $E_n$ is semistable, there exist $M = \ell n$ and $\vec \pi \in (S_M)^n$ such that 
$P_{M, \vec \pi}(E_n) \ne 0$ (Cor.~\ref{cor1}). %(Note that here the dimension is $d = n$.) 
By Lemmas~\ref{pgl} and \ref{16} we have 
$$
P_{M, \vec \pi}(D) = P_{M, \vec \pi}(E_n) \cdot \det(B_1)^{\ell} \cdots \det(B_n)^{\ell} \ne 0.
$$
On the other hand, let us check the expansion of this polynomial, which is given by
\begin{align*}
P_{M, \vec{\pi}}(D) &= \sum_{J_1,\ldots, J_n\, :\, [M] \to [n]}\, \prod_{k = 1}^n \varepsilon(J_k \circ \pi_k) \prod_{i = 1}^M D(J_1(i), \ldots, J_n(i)) \\
&= \sum_{J_1,\ldots, J_n\, :\, [M] \to [n]}\, \prod_{k = 1}^n \varepsilon(J_k \circ \pi_k) \prod_{i = 1}^M \det(B_1[J_1(i)], \ldots, B_n[J_n(i)]).
\end{align*}
Since $P_{M, \vec{\pi}}(D) \ne 0$, at least one term in this expansion is also nonzero, which will give a desired arrangement. Indeed, if 
$$
\prod_{k = 1}^n \varepsilon(J_k \circ \pi_k) \prod_{i = 1}^M \det(B_1[J_1(i)], \ldots, B_n[J_n(i)]) \ne 0
$$
then we can arrange the columns of $B_1,\ldots, B_n$ into $n \times M$ matrix $A$ w.r.t. the maps $J_1,\ldots, J_n : [M] \to [n]$ such that the $i$-th column of $A$ has the entries $B_1[J_1(i)], \ldots, B_n[J_n(i)]$ of the corresponding columns of $B_1,\ldots, B_n$. Since $\det(B_1[J_1(i)], \ldots, B_n[J_n(i)]) \ne 0$ they are all bases as needed. The rows of $A$ also satisfy the needed property, i.e. each entry appears exactly $\ell$ times, since $\varepsilon(J_k \circ \pi_k) \ne 0$ for all $k = 1,\ldots, n$ which is clear from the definition of the sign $\varepsilon(J)$.  %$B_{k}[J_{k}(i)]$???
\qed

\begin{remark}
From Remark~\ref{lbd}, we can see that an upper bound on the multiplicity $\ell = M/n$ is large, it gives $\ell \le n^{n^3}$. %, whereas Rota's basis conjecture asks $\ell = 1$.  
\end{remark}

%\section{On the Alon--Tarsi conjecture}

\section{Concluding remarks}
\subsection{} 
As discussed in \cite{hr}, Rota's basis conjecture is related %to some algebraic questions in invariant theory, particularly related 
to certain conjectured polynomial identities originating from invariant theory. 
In \cite{rot10}, Rota leaves the following interesting remark %expressing his feelings 
on his conjecture: 
%Let us quote from Rota \cite{rot10}:
\begin{quote}
``\,%Behind this conjecture lie certain identities from invariant theory which remain unproved, and which must be passed over in silence. [...] 
I would feel crushed if the basis conjecture were to be settled by methods other than some new insight in the algebra of invariant theory.\,"
\end{quote}

\subsection{} %The question we can ask now %after Theorem~\ref{one} 
%is the following. 
It is reasonable to ask if Conjecture~\ref{rota} can now be completed by a combinatorial argument.
Is it possible to transform a matrix $A$  satisfying the conditions of Theorem~\ref{one} (e.g. via some exchange operations), so that we can choose $n$ columns satisfying the conditions of Conjecture~\ref{rota}? For instance, one matrix (obtained after transformations) which resolves the problem, is a matrix whose every column is repeated $\ell$ times.  

\subsection{} %{On the Alon--Tarsi conjecture} 
The Alon--Tarsi conjecture on latin squares \cite{at} can be formulated that $P_{n, \vec \pi}(E_n) \ne 0$ for even $n$ (here the minimal invariant function $P_{n, \vec \pi}$ coincides (up to a sign) with Cayley's first hyperdeterminant). Hence our result that $E_n$ is semistable and $P_{M, \vec \pi}(E_n) \ne 0$ for some $M$ divisible by $n$, can be viewed as an analogue of this conjecture. This result can also be formulated in terms of certain Latin-type matrices. 

\subsection{}
Finally, it is tempting to draw a distant similarity in the approach to Klyachko's theorem on asymptotic saturation of Littlewood--Richardson coefficients \cite{kly} whose proof also relied on geometric invariant theory.

\section*{Acknowledgements}
I am grateful to Alimzhan Amanov for useful comments and many interesting conversations. 

%\newpage

\end{document}